\documentclass[12pt]{article}
\usepackage{latexsym,amssymb,float,amsmath,amsthm,enumerate,cite,geometry}
\newtheorem{theorem}{Theorem}

\newtheorem{lemma}[theorem]{Lemma}
\newtheorem{corollary}[theorem]{Corollary}

\usepackage{lineno}
\usepackage{setspace}
\allowdisplaybreaks

\begin{document}
\onehalfspace

\title{Revisiting Extremal Graphs Having No Stable Cutsets}
\author{
Johannes Rauch\and 
Dieter Rautenbach}
\date{}

\maketitle
\vspace{-10mm}
\begin{center}
Ulm University, Institute of Optimization and Operations Research\\ 
89081 Ulm, Germany\\
\texttt{$\{$johannes.rauch,dieter.rautenbach$\}$@uni-ulm.de}
\end{center}

\begin{abstract}
Confirming a conjecture posed by Caro,
it was shown by Chen and Yu 
that every graph $G$ with $n$ vertices and at most $2n-4$ edges
has a stable cutset,
which is a stable set of vertices whose removal disconnects the graph.
Le and Pfender showed that all graphs with $n$ vertices and $2n-3$ edges 
without stable cutset
arise recursively glueing together triangles and triangular prisms 
along an edge or triangle.
Le and Pfender's proof contains a gap, 
which we fill in the present article.\\[3mm]
\end{abstract}

\bigskip

\bigskip

\bigskip

\bigskip

\section{Introduction}\label{sec1}

We consider only finite, simple, and undirected graphs and 
refer to \cite{lepf} for further notational details.
A {\it stable cutset} in a graph $G$ is a stable set $S$ of vertices of $G$
for which $G-S$ is disconnected.
By an elegant inductive argument,
Chen and Yu \cite{chyu} showed the following result
confirming a conjecture by Caro.

\begin{theorem}[Chen and Yu \cite{chyu}]\label{theoremchyu}
If $G$ is a graph with $n$ vertices and at most $2n-4$ edges, 
then $G$ contains a stable cutset.
\end{theorem}
Le and Pfender \cite{lepf} gave an elegant structural characterization 
of the graphs $G$ of order $n$ with $2n-3$ edges
that do not contain a stable cutset, 
cf.~Theorem \ref{theoremlepf} below.
Our present goal is to fill a gap 
in the original proof given by Le and Pfender for Theorem \ref{theoremlepf}. 
Stable cutsets were considered in a number of publications
concerning structural refinements,
algorithmic complexity, 
tractable cases, 
fixed parameter tractability, and their relation to perfect graphs
\cite{BrDrLeSz,MR1936948,ch,cofo,klfi,le_mosca_mueller_2008,krle,LeRa,MaOsRa,raraso,tu}.

\section{Graphs $G$ with $2n(G)-3$ edges and no stable cutset}

We recall definitions and results from \cite{lepf},
formulate the main result Theorem \ref{theoremlepf},
and provide a proof, 
in which we also explain the gap in the original proof.

If $G$ and $H$ are two graphs 
and $G\cap H$ is isomorphic to $K_2$ or $K_3$,
then $G\cup H$ is said to arise from $G$ and $H$ 
by an {\it edge identification} or a {\it triangle identification}, 
respectively.
Le and Pfender \cite{lepf} 
define the class ${\cal G}_{sc}$ of graphs recursively as follows:
\begin{itemize}
\item $K_3,\overline{C_6}\in {\cal G}_{sc}$.
\item If $G,H\in {\cal G}_{sc}$ and $G\cap H$ is isomorphic to $K_2$ or $K_3$,
then $G\cup H\in {\cal G}_{sc}$.
\end{itemize}
They already observe that 
$H$ may be restricted to $\left\{ K_3,\overline{C_6}\right\}$
without changing ${\cal G}_{sc}$.

For a positive integer $k$, let $[k]$ denote the set of positive integers at most $k$.

A {\it generating sequence} for a graph $G$ is a sequence $(G_1,\ldots,G_k)$ 
for some positive integer $k$ such that
\begin{itemize}
\item for every $i\in [k]$,
the graph $G_i$ is isomorphic to $K_3$ or $\overline{C_6}$,
\item for every $i\in [k-1]$,
the graph $G_{\leq i}\cap G_{i+1}$ 
is isomorphic to $K_2$ or $K_3$,
where $G_{\leq i}=G_1\cup \ldots \cup G_i$, and
\item $G=G_{\leq k}$.
\end{itemize}
A simple inductive argument using the recursive definition of ${\cal G}_{sc}$
implies that every graph with a generating sequence belongs to ${\cal G}_{sc}$.
Le and Pfender's mentioned observation is that every graph in ${\cal G}_{sc}$
has a generating sequence.
We need the following strengthening of this statement.

\begin{lemma}\label{lemma1}
Every graph $G$ in ${\cal G}_{sc}$ has a generating sequence $(G_1,\ldots,G_k)$.
Furthermore, for every $i\in [k]$, 
the graph $G$ has a generating sequence $(H_1,\ldots,H_k)$ with $H_1=G_i$.
\end{lemma}
\begin{proof}
The proof is by induction on the order of $G$ 
using the original recursive definition of ${\cal G}_{sc}$.
If $G$ is isomorphic to $K_3$ or $\overline{C_6}$,
then $(G_1)=(G)$ is a generating sequence for $G$,
and the second statement is trivially true.

Now, let $G=G^{(1)}\cup G^{(2)}$ be such that 
$G^{(1)},G^{(2)}\in {\cal G}_{sc}$ are proper subgraphs of $G$
and $G^{(1)}\cap G^{(2)}$ is isomorphic to $K_2$ or $K_3$,
that is, the two graphs $G^{(1)}$ and $G^{(2)}$
share exactly two or three vertices that form a clique in both graphs.
By induction,
the graph $G^{(1)}$ has a generating sequence $\left(G^{(1)}_1,\ldots,G^{(1)}_k\right)$
and 
the graph $G^{(2)}$ has a generating sequence $\left(G^{(2)}_1,\ldots,G^{(2)}_\ell\right)$.
Let $i\in [\ell]$ be such that 
the edge or triangle $G^{(1)}\cap G^{(2)}$ is a subgraph of $G^{(2)}_i$;
the existence of such an index $i$ 
follows immediately from the definition of ${\cal G}_{sc}$.
By the second statement,
the graph $G^{(2)}$ has a generating sequence $\left(H^{(2)}_1,\ldots,H^{(2)}_\ell\right)$ 
with $H^{(2)}_1=G^{(2)}_i$.
Now, the sequence $\left(G^{(1)}_1,\ldots,G^{(1)}_k,H^{(2)}_1,\ldots,H^{(2)}_\ell\right)$ 
is a generating sequence $(G_1,\ldots,G_{k+\ell})$ for $G$.

For the second statement,
it remains to show that, for every $j\in [k+\ell]$,
the graph $G$ has a generating sequence starting with $G_j$.
If $j\in [k]$, then, 
by the second statement,
the graph $G^{(1)}$ has a generating sequence $\left(H^{(1)}_1,\ldots,H^{(1)}_k\right)$ 
with $H^{(1)}_1=G^{(1)}_j=G_j$,
and the sequence
$\left(H^{(1)}_1,\ldots,H^{(1)}_k,H^{(2)}_1,\ldots,H^{(2)}_\ell\right)$
is a generating sequence for $G$ starting with $G_j$.
Now, let $j\in [k+\ell]\setminus [k]$.
By the second statement,
the graph $G^{(2)}$ has a generating sequence $\left(I^{(2)}_1,\ldots,I^{(2)}_\ell\right)$ 
with $I^{(2)}_1=H^{(2)}_{j-k}=G_j$.
Similarly as above, 
there is some $i\in [\ell]$ such that 
the edge or triangle $G^{(1)}\cap G^{(2)}$ is a subgraph of $G^{(1)}_i$.
By the second statement,
the graph $G^{(1)}$ has a generating sequence $\left(H^{(1)}_1,\ldots,H^{(1)}_k\right)$ 
with $H^{(1)}_1=G^{(1)}_i$.
Now, the sequence
$\left(I^{(2)}_1,\ldots,I^{(2)}_\ell,H^{(1)}_1,\ldots,H^{(1)}_k\right)$
is a generating sequence for $G$ starting with $G_j$.
This completes the proof.
\end{proof}
From the main result of Chen and Yu \cite{chyu},
Le and Pfender \cite{lepf} deduce the following.
\begin{corollary}[Le and Pfender, Corollary 3 in \cite{lepf}]\label{corollary1}
If $G$ is a graph of order $n$ with at most $2n-4$ edges and 
$x$ is not the only cut vertex in $G$,
then $G$ has a stable cutset not containing $x$.
\end{corollary}
The following is the main result from \cite{lepf}.

\begin{theorem}[Le and Pfender, Theorem 5 in \cite{lepf}]\label{theoremlepf}
If $G$ is a graph of order $n$ with at most $2n-3$ edges, 
then $G$ has a stable cutset or belongs to ${\cal G}_{sc}$.
\end{theorem}
\begin{proof} 
For a proof by contradiction as in \cite{lepf},
we assume that $G$ is a counterexample of minimum order $n$.
The following properties of $G$ are deduced in \cite{lepf},
where we use the same numbering of the claims as in \cite{lepf}:\\[3mm]
{\bf Claim 6.} {\it $G$ has exactly $2n-3$ edges.}\\[3mm]
{\bf Claim 7.} {\it Every vertex of $G$ lies in a triangle.}\\[3mm]
{\bf Claim 8.} {\it $G$ contains no $K_2$-cutset or $K_3$-cutset.}\\[3mm]
{\bf Claim 9.} {\it $G$ is $3$-connected.}\\[3mm]
{\bf Claim 10.} {\it $G$ contains no $3$-edge matching cut.}\\[3mm]
{\bf Claim 11.} {\it $G$ contains no $K_4^-$.}\\[3mm]
{\bf Claim 12.} {\it For every two non-adjacent vertices $x$ and $y$, 
we have $|N_G(x)\cap N_G(y)|\leq 2$.}\\[3mm]
{\bf Claim 13.} {\it $G$ contains no $P_3$-cutset.}

\bigskip

The gap in the argument lies in the proof of the following claim.\\[3mm]
{\bf Claim 14.} {\it In every triangle, at least two vertices belong to other triangles as well.}
\begin{proof}[Proof of Claim 14.]
For a proof by contradiction, we assume that $xy_0z_0$ is a triangle in $G$
and that $y_0$ and $z_0$ lie in no other triangles in $G$.
Since $(N_G(y_0)\cup N_G(z_0))\setminus \{ y_0,z_0\}$ is not a stable cutset,
there are neighbors $y_1$ of $y_0$ and $z_1$ of $z_0$ 
such that $y_1$ and $z_1$ are adjacent.
By Claim 12, the vertices $y_1$ and $z_0$ are the only common neighbors 
of $y_0$ and $z_1$.
Let the graph $G'$ arise from $G$ by identifying the vertices $y_0$ and $z_1$ 
to form the vertex $v$.
The order of $G'$ is $n-1$ and its size is $2(n-1)-3$.
Since every stable cutset in $G'$ is also a stable cutset in $G$,
it follows that $G'$ has no stable cutset.
Now, the choice of $G$ implies that $G'\in {\cal G}_{sc}$.

\begin{quote}
{\it Explanation of the gap:\\
At this point, Le and Pfender correctly show that $G'$ does not contain a $3$-edge matching cut.
From that they incorrectly deduce that $G'$ can be built by starting with a triangle 
and recursively glueing on triangles along an edge,
that is, that $G'$ is a so-called $2$-tree.
Clearly, edge or triangle identifications with copies of $\overline{C_6}$
during the construction of $G'$ create $3$-edge matching cuts in intermediate graphs.
Nevertheless, subsequent further identifications in the construction of $G'$ 
can eliminate these cuts.}
\end{quote}

By Lemma \ref{lemma1}, the graph $G'$ 
has a generating sequence ${\cal S}=(G'_1,\ldots,G'_\ell)$ 
such that $G'_1$ contains the triangle $xvz_0$.
Possibly by inserting the triangle $xvz_0$ within ${\cal S}$ before $G'_1$,
we may assume that $G'_1$ equals the triangle $xvz_0$.
If $G'_{i+1}$ is isomorphic to $\overline{C_6}$
and $G'_{\leq i}\cap G'_{i+1}$ is isomorphic to $K_2$,
then we may assume that edge $ab$ common to $G'_{\leq i}$ and $G'_{i+1}$
belongs to the $3$-edge matching cut of $G'_{i+1}$;
otherwise, the edge $ab$ belongs to a triangle $abc$ in $G'_{i+1}$
with $c\not\in V(G'_{\leq i})$,
and we can replace $G'_{i+1}$ within the generating sequence ${\cal S}$
by $abc,G'_{i+1}$,
that is, we consider the alternative generating sequence 
$(G'_1,\ldots,G'_{i},abc,G'_{i+1},\ldots,G'_\ell)$, 
where we first form 
an edge identification along $ab$ with the triangle $abc$ and then 
a triangle identification along $abc$ with $G'_{i+1}$. 
Subject to these restrictions, 
we assume that the generating sequence ${\cal S}$
is chosen in such a way that the smallest index $k$ 
with $y_1\in V(G'_k)$ is as small as possible.

By Claim 8, 
the vertex $v$ is involved in every edge or triangle identification within ${\cal S}$,
more precisely, the vertex $v$ belongs to each graph $G'_i$ within ${\cal S}$.\\[3mm]
{\bf Claim 14a.} {\it The graph $G'_{\leq k}$ does not have a stable set $X_k$ 
containing $y_1$ and $z_0$ and not containing $v$
such that,
if $v$ lies on a cycle $C$ in $G'_{\leq k}-X_k$, then, in the graph $G$, 
the two neighbors of $v$ on $C$
are adjacent to $z_1$ and non-adjacent to $y_0$.}
\begin{proof}[Proof of Claim 14a.]
For a proof by contradiction, we assume the existence of such a set $X_k$.
By an inductive argument along the generating sequence starting at $G'_{\leq k}$,
which is the first graph containing the two vertices $z_0$ and $y_1$,
we show that $X_k$ can be extended to sets 
$X_k\subseteq X_{k+1}\subseteq \ldots \subseteq X_\ell$
such that, for every $i\in [\ell]\setminus [k-1]$,
the set $X_i$ is a stable set in $G'_{\leq i}$,
contains $y_1$ and $z_0$,
does not contain $v$, and, 
if $v$ lies on a cycle $C$ in $G'_{\leq i}-X_i$, then, in the graph $G$, 
the two neighbors of $v$ on $C$
are adjacent to $z_1$ and non-adjacent to $y_0$.

For $i=k$, the statement is our assumption.

Now, let $i>k$.
If $G'_i$ is a triangle $vab$, where $v$ and $b$ belong to $G'_{\leq i-1}$, then 
$$X_i=
\begin{cases}
X_{i-1},&\mbox{ if $b\in X_{i-1}$ and}\\
X_{i-1}\cup \{ a\},&\mbox{ otherwise}
\end{cases}$$
has the desired properties.
If $G'_i$ is isomorphic to $\overline{C_6}$ with the two triangles
$uvw$ and $abc$ and the $3$-edge matching cut $\{au,bv,cw\}$,
where $v$ and $b$ belong to $G'_{\leq i-1}$, then 
$$X_i=
\begin{cases}
X_{i-1}\cup \{ w\},&\mbox{ if $b\in X_{i-1}$ and}\\
X_{i-1}\cup \{ a,w\},&\mbox{ otherwise}
\end{cases}$$
has the desired properties.
Finally, 
if $G'_i$ is isomorphic to $\overline{C_6}$ with the two triangles
$uvw$ and $abc$ and the $3$-edge matching cut $\{au,bv,cw\}$,
where $u$, $v$, and $w$ belong to $G'_{\leq i-1}$, then 
$X_i=X_{i-1}\cup \{ b\}$
has the desired properties.
Note that in the final case, 
if $X_{i-1}$ does not contain either $u$ or $w$,
then, in the graph $G$,
these two vertices are adjacent to $z_1$ and non-adjacent to $y_0$.
This completes the inductive argument.

Now, the set $X_\ell$ is also a stable set in the graph $G$
containing $y_1$ and $z_0$ and not containing $y_0$ and $z_1$.
Suppose, for a contradiction, 
that $y_0$ and $z_1$ lie in the same component of $G-X_\ell$.
Since the two common neighbors of $y_0$ and $z_1$ belong to $X_\ell$,
a path in $G-X_\ell$ between $y_0$ and $z_1$ has length at least three,
and the vertex $v$ lies on a cycle $C$ in $G'-X_\ell$
such that, in the graph $G$,
one of the two neighbors of $v$ on $C$ is adjacent to $y_0$
and 
the other one of the two neighbors of $v$ on $C$ is adjacent to $z_1$,
which is a contradiction. 
Hence, the set $X_\ell$ is a stable cutset in $G$,
which is a contradiction and completes the proof of the subclaim.
\end{proof}
If the vertex $z_0$ is involved in any edge identification within ${\cal S}$,
then, in the graph $G$,
the set $\{ y_0,z_0,z_1\}$ is a $P_3$-cutset,
contradicting Claim 13.
Hence, the vertex $z_0$ is not involved in any edge identification within ${\cal S}$.

Our next goal is to construct an induced path $P:y_1y_2\ldots y_{p-1}y_p$ in $G'_{\leq k}-v$
starting in the neighbor $y_1$ of $v$,
ending in $(y_{p-1},y_p)=(x,z_0)$, and 
containing all neighbors of $v$ in $G'_k$.
We construct this path inductively following the generating sequence
backwards from $G'_k$ down to $G'_1$
starting in $y_1$.
Suppose that, for some $i\in [k]\setminus \{ 1\}$, 
we have already constructed an induced path $y_1\ldots y_j$ 
starting in $y_1$ such that 
\begin{itemize}
\item the set $\{ y_1,\ldots,y_{j-1}\}$ 
is a subset of $V(G'_{\leq k})\setminus V(G'_{\leq i})$,
\item the set $\{ y_1,\ldots,y_{j-1}\}$ 
contains all neighbors of $v$ in $G'_{\leq k}$ that do not belong to $G'_{\leq i}$,
\item $y_j$ is a neighbor of $v$,
\item $y_j$ is the only vertex of $y_1\ldots y_j$ that belongs to $G'_i$, and
\item $y_j\not\in G'_{\leq i-1}$.
\end{itemize}
Initially, this holds for $i=k$ and $j=1$,
and the inductive construction is such that these properties are maintained.
If $G'_i$ is a triangle $vab$, where $v$ and $b$ belong to $G'_{\leq i-1}$
and $y_j=a$, then the choice of the generating sequence ${\cal S}$
implies that the vertex $b$ belongs to $G'_{i-1}$ but not to $G'_{\leq i-2}$,
where $G'_{\leq 0}$ is the empty graph.
Now, setting $y_{j+1}=b$ has the desired properties (for $i$ replaced by $i-1$).
If $G'_i$ is isomorphic to $\overline{C_6}$ with the two triangles
$uvw$ and $abc$ and the $3$-edge matching cut $\{au,bv,cw\}$,
where $v$ and $b$ belong to $G'_{\leq i-1}$
and $y_j=u$, then the choice of the generating sequence ${\cal S}$
implies that the vertex $b$ belongs to $G'_{i-1}$ but not to $G'_{\leq i-2}$.
Now, setting 
$y_{j+1}=w$,
$y_{j+2}=c$, and
$y_{j+3}=b$
has the desired properties (for $i$ replaced by $i-1$).
See Figure \ref{fig1} for an illustration.

\begin{figure}[H]
\centering
{\footnotesize 
\unitlength 0.9mm 
\linethickness{0.4pt}
\ifx\plotpoint\undefined\newsavebox{\plotpoint}\fi 
\begin{picture}(106,38)(0,0)
\put(40,8){\circle*{1}}
\put(60,8){\circle*{1}}
\put(80,8){\circle*{1}}
\put(40,28){\circle*{1}}
\put(60,28){\circle*{1}}
\put(80,28){\circle*{1}}
\put(40,28){\line(0,-1){20}}
\put(60,28){\line(0,-1){20}}
\put(80,28){\line(0,-1){20}}
\put(60,28){\line(-1,0){20}}
\put(80,28){\line(-1,0){20}}
\put(60,8){\line(-1,0){20}}
\put(80,8){\line(-1,0){20}}
\put(40,31){\makebox(0,0)[cc]{$v$}}
\put(96,28){\makebox(0,0)[cc]{$\ldots$}}
\put(100,28){\line(1,0){3}}
\put(103,28){\circle*{1}}
\put(106,28){\makebox(0,0)[lc]{$y_1$}}
\put(23,38){\makebox(0,0)[cc]{$G'_{\leq i-1}$}}
\put(60,37){\makebox(0,0)[cc]{$G'_i$}}
\put(23.5,18){\oval(45,34)[]}
\qbezier(40,28)(60,18)(80,28)
\qbezier(80,8)(60,18)(40,8)
\put(80,31){\makebox(0,0)[cc]{$y_j=u$}}
\put(80,5){\makebox(0,0)[cc]{$a$}}
\put(60,31){\makebox(0,0)[cc]{$y_{j+1}=w$}}
\put(60,5){\makebox(0,0)[cc]{$y_{j+2}=c$}}
\put(40,5){\makebox(0,0)[rc]{$y_{j+3}=b$}}
\put(56,18){\oval(68,32)[]}
\put(80,28){\line(1,0){12}}
\end{picture}}
\caption{Definition of $P$ for an edge identification with $\overline{C_6}$.}
\label{fig1}
\end{figure}
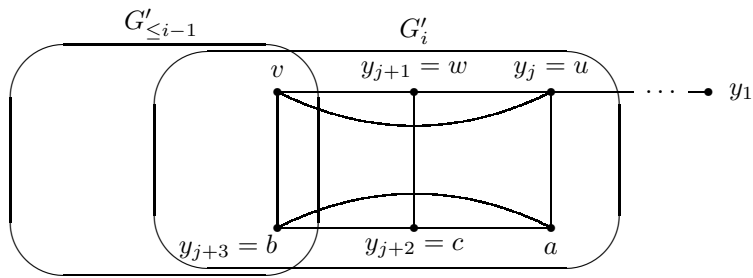
Finally, 
if $G'_i$ is isomorphic to $\overline{C_6}$ with the two triangles
$uvw$ and $abc$ and the $3$-edge matching cut $\{au,bv,cw\}$,
where $u$, $v$, and $w$ belong to $G'_{\leq i-1}$
and $y_j=b$,
then the choice of the generating sequence ${\cal S}$
implies that at least one of the two vertices $u$ and $w$, say $u$, 
belongs to $G'_{i-1}$ but not to $G'_{\leq i-2}$.
Now, setting 
$y_{j+1}=a$ and
$y_{j+2}=u$
has the desired properties (for $i$ replaced by $i-1$).
See Figure \ref{fig2} for an illustration.

\begin{figure}[H]
\centering
{\footnotesize 
\unitlength 0.9mm 
\linethickness{0.4pt}
\ifx\plotpoint\undefined\newsavebox{\plotpoint}\fi 
\begin{picture}(92,38)(0,0)
\put(40,8){\circle*{1}}
\put(60,8){\circle*{1}}
\put(40,18){\circle*{1}}
\put(60,18){\circle*{1}}
\put(40,28){\circle*{1}}
\put(60,28){\circle*{1}}
\put(40,28){\line(0,-1){20}}
\put(60,28){\line(0,-1){20}}
\qbezier(40,28)(50,18)(40,8)
\qbezier(60,28)(50,18)(60,8)
\put(60,28){\line(-1,0){20}}
\put(60,8){\line(-1,0){20}}
\put(60,18){\line(-1,0){20}}
\put(40,31){\makebox(0,0)[cc]{$v$}}
\put(60,31){\makebox(0,0)[cc]{$y_j=b$}}
\put(60,5){\makebox(0,0)[cc]{$c$}}
\put(39,18){\makebox(0,0)[rc]{$y_{j+2}=u$}}
\put(41,5){\makebox(0,0)[rc]{$y_{j+3}=w$}}
\put(61,18){\makebox(0,0)[lc]{$a=y_{j+1}$}}
\put(82,28){\makebox(0,0)[cc]{$\ldots$}}
\put(86,28){\line(1,0){3}}
\put(89,28){\circle*{1}}
\put(92,28){\makebox(0,0)[lc]{$y_1$}}
\put(23,38){\makebox(0,0)[cc]{$G'_{\leq i-1}$}}
\put(50,37){\makebox(0,0)[cc]{$G'_i$}}
\put(49.5,18){\oval(55,32)[]}
\put(23.5,18){\oval(45,34)[]}
\put(60,28){\line(1,0){19}}
\end{picture}}
\caption{Definition of $P$ for a triangle identification with $\overline{C_6}$.}\label{fig2}
\end{figure}
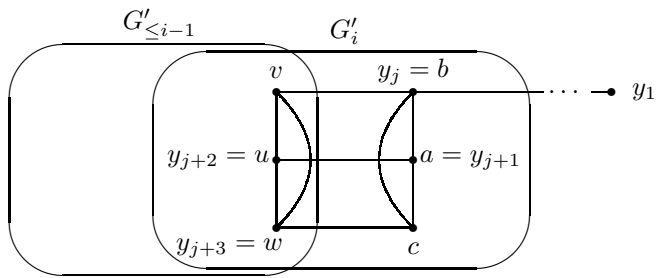
Note that in this final case, 
the next step of the construction of $P$ ensures $y_{j+3}=w$.
If $i=1$, then, since $z_0$ is not involved in any edge identification,
we may assume that $y_j=x$.
Now, setting $\ell=j+1$ and $y_\ell=z_0$ yields $P$ as desired.
This completes the construction of $P$.

\bigskip

If $p$ is odd, then let $X_k=\{ y_1,y_3,y_5,\ldots,y_p\}$.
By the construction of $P$, 
the set $X_k$ is stable,
contains $y_1$ and $y_p=z_0$, 
does not contain $v$, and 
intersects every triangle as well as every induced $C_4$ in $G'_k$ that contains $v$,
contradicting Claim 14a.
Hence, it follows that $p$ is even.

First, suppose that the generating sequence ${\cal S}'=(G'_1,\ldots,G'_k)$ 
of $G'_{\leq k}$
involves an edge identification with $\overline{C_6}$.
Let $i\in [k]$ be the largest index such that $G'_i$ is isomorphic to $\overline{C_6}$
and $G'_{\leq i-1}\cap G'_i$ is the edge $vb$.
Let $j$ be the smallest index with $y_j\in V(G'_i)$.
See Figure \ref{fig1} for an illustration.
If $j$ is odd, then let
$X_k=\{ y_1,y_3,\ldots,y_j\}\cup\{ y_{j+3},y_{j+5},\ldots,y_p\}$, and
if $j$ is even, then let
$X_k=\{ y_1,y_3,\ldots,y_{j+1}\}\cup \{ a\}\cup\{ y_{j+4},y_{j+6},\ldots,y_p\}$.
Again,
the set $X_k$ is stable,
contains $y_1$ and $y_p=z_0$, 
does not contain $v$, and 
intersects every triangle as well as every induced $C_4$ in $G'_k$ that contains $v$,
contradicting Claim 14a.
Hence, the generating sequence ${\cal S}'$ 
involves no edge identification with $\overline{C_6}$.

Next, suppose that the generating sequence ${\cal S}'=(G'_1,\ldots,G'_k)$ 
of $G'_{\leq k}$
involves a triangle identification with $\overline{C_6}$.
Let $i\in [k]$ be the largest index such that $G'_i$ is isomorphic to $\overline{C_6}$
and $G'_{\leq i-1}\cap G'_i$ is the triangle $uvw$.
Let $j$ be the smallest index with $y_j\in V(G'_i)$.
See Figure \ref{fig2} for an illustration.
Recall that $y_{j+3}=w$ in this case.
If $j$ is odd, then let
$X_k=\{ y_1,y_3,\ldots,y_j\}\cup\{ y_{j+3},y_{j+5},\ldots,y_p\}$, and
if $j$ is even, then let
$X_k=\{ y_1,y_3,\ldots,y_{j-1}\}\cup \{ c\}\cup\{ y_{j+2},y_{j+4},\ldots,y_p\}$.
Again,
the set $X_k$ is stable,
contains $y_1$ and $y_p=z_0$, 
does not contain $v$, and 
intersects every triangle as well as every induced $C_4$ in $G'_k$ that contains $v$,
contradicting Claim 14a.
Hence, the generating sequence ${\cal S}'$ 
involves no triangle identification with $\overline{C_6}$.

Altogether, it follows that each $G'_i$ in ${\cal S}'$ is isomorphic to $K_3$,
which implies that $k=p+1$ is odd.\\[3mm]
{\bf Claim 14b.} {\it $y_0$ is adjacent to each vertex in $\{ y_1,y_3,\ldots,y_{p-1}\}$.}
\begin{proof}[Proof of Claim 14b.]
Suppose, for a contradiction, 
that $y_0$ is not adjacent to $y_j$ for some odd index $j$ in $[p]$.
Since $y_0$ is adjacent to $y_1$ and $y_{p-1}=x$,
this implies $p\geq 6$.
Choosing $j$ as the smallest odd index such that $y_0$ is not adjacent to $y_j$,
we obtain that the vertices in $\{ y_1,y_3,\ldots,y_{j-2}\}$ are all adjacent to $y_0$.
Since $xy_0z_0$ is the only triangle in $G$ that contains $y_0$,
the vertex $y_{j-1}$ is adjacent to $z_1$ and non-adjacent to $y_0$.
Now, 
the set $X_k=\{ y_1,y_3,\ldots,y_{j-2}\}\cup \{ y_{j+1},y_{j+3},\ldots,y_p\}$
contradicts Claim 14a,
which completes the proof.
\end{proof}
Since $xy_0z_0$ is the only triangle in $G$ that contains $y_0$,
Claim 14b implies that 
$z_1$ is adjacent to each vertex in $\{ y_2,y_4,\ldots,y_p\}$.\\[3mm]
{\bf Claim 14c.} {\it $p=4$.}
\begin{proof}[Proof of Claim 14c.]
Suppose, for a contradiction, that $p\geq 6$.
Let the graph $G''$ arise from $G$ by identifying the vertices $z_0$ and $y_1$ 
to form a vertex $v''$.
Similarly, as for $G'$, it follows that $G''\in {\cal G}_{sc}$.
Nevertheless, the graph $G''$ contains an induced cycle 
$v''y_0y_{p-3}y_{p-2}z_1v''$ of length $5$, 
contradicting $G''\in {\cal G}_{sc}$.
\end{proof}
At this point, the subgraph of $G$ 
induced by $\{ x,y_0,z_0,y_1,z_1,y_2\}=\{ y_0,z_1\}\cup V(P)$
is isomorphic to $\overline{C_6}$ 
with the two triangles being $xy_0z_0$ and $y_1y_2z_1$.
Let $G'''$ arise from $G$ by identifying the vertices of $P:y_1y_2xz_0$ 
to form a vertex $v'''$.
The order of $G'''$ is $n-3$ and its size is at most $(2n-3)-7=2(n-3)-4$.
If $v'''$ is not the only cut vertex in $G'''$, then, by Corollary \ref{corollary1},
the graph $G'''$ has a stable cutset not containing $v'''$,
which is also a stable cutset in $G$.
Hence, it follows that $v'''$ is the only cut vertex in $G'''$.
Since $v$ is involved in every edge or triangle identification within ${\cal S}$,
it follows that all vertices added by the identifications with $G'_{k+1},\ldots,G'_{\ell}$
belong to the same component of $G'''$ as $y_0$ or $z_1$.
This implies that $G'''-v'''$ has exactly two components,
one component $C_0$ containing $y_0$ and 
the other component $C_1$ containing $z_1$.
Furthermore, for every $i\in [\ell]\setminus [k]$,
all vertices in $V(G'_i)\setminus V(G'_{\leq i-1})$ 
belong to either $C_0$ or $C_1$.
Since $G$ is not isomorphic to $\overline{C_6}$,
it follows that $\ell>k$.
In the graph $G'_{\leq k+1}$,
the set $X=V(G'_{\leq k})\cap V(G'_{k+1})$ is a $K_2$-cutset or a $K_3$-cutset.
In the subgraph of $G$ 
induced by $\{ x,y_0,z_0,y_1,z_1,y_2\}\cup \big(V(G'_{k+1})\setminus \{ v\}\big)$, 
the set $X$ is 
a $K_2$-cutset or
a $K_3$-cutset or
a $P_3$-cutset.
Furthermore, by the structural properties observed above,
the set $X$ is still a cutset in $G$,
which contradicts Claims 8 and 13.

This final contradiction completes the proof of Claim 14.
\end{proof}
At this point the entire proof can be finished exactly as in \cite{lepf}.
\end{proof}

\noindent {\bf Acknowledgement} 
We thank 
Van Bang Le (Universität Rostock)
and 
Florian Pfender (University of Colorado Denver)
for excellent discussions concerning Theorem \ref{theoremlepf} and its proof.
Since their result is just too beautiful to be false,
we are quite happy having closed the gap.
This work was partially supported by the Deutsche Forschungsgemeinschaft 
(DFG, German Research Foundation) -- project number 545935699.

\end{document}